\documentclass{article}
\usepackage{graphicx}
\usepackage{amsmath,amssymb}

\newtheorem{theorem}{Theorem}[section]

\newtheorem{proposition}[theorem]{Proposition}

\newtheorem{lemma}[theorem]{Lemma}
\newtheorem{corollary}[theorem]{Corollary}

\newtheorem{example}[theorem]{Example}

\begin{document}


\title{Minimization of Constrained Quadratic forms in Hilbert Spaces}
\date{}
\author{Dimitrios Pappas\\Department of Statistics,\\Athens University of
Economics and Business,\\76 Patission Str, 10434, Athens, Greece
\\(dpappas@aueb.gr, pappdimitris@gmail.com)} \maketitle
\begin{abstract}A common optimization problem is the minimization of a symmetric positive
definite quadratic form $\langle x,Tx\rangle$ under linear
constrains. The solution to this problem may be given using the
Moore-Penrose inverse matrix. In this work we extend this result to
infinite dimensional complex Hilbert spaces, making use of the
generalized inverse of an operator. A generalization is given for
positive diagonizable and arbitrary positive operators,
 not necessarily invertible, considering as constraint a singular
 operator. In particular, when $T$ is positive semidefinite, the minimization
 is considered for all vectors belonging to $\mathcal{N}(T)^\perp$.
\end{abstract}
\textit{Keywords}: Quadratic form, Constrained Optimization, Moore-Penrose inverse,\\
Positive operator
\\\textit{2010 Mathematics Subject Classification}: 47A05, 47N10, 15A09.

\section{Introduction}
Quadratic forms have played a central role in the history of
mathematics, in both the finite and the infinite dimensional case.
Many authors have studied problems on minimizing (or maximizing)
quadratic forms under various constrains, such as vectors
constrained to lie within the unit simplex (Broom \cite{broom}).  A
similar result is the minimization of a more general case of a
quadratic form defined in a finite-dimensional real Euclidean space
under linear constraints (see e.g. La Cruz \cite{cruz},  Manherz and
Hakimi \cite{hak}), with many applications in network analysis and
control theory. In a classical book of Optimization Theory by
Luenberger \cite{luen}, various similar optimization problems are
presented, for both finite and infinite dimensions.\\In the field of
applied mathematics, a strong interest is shown in applications of
the generalized inverse of matrices or operators. Various types of
generalized inverses are used whenever a matrix/ operator is
singular, in many fields of both computational and also theoretical
aspects. An application of the Moore-Penrose inverse in the finite
dimensional case, is the minimization of a symmetric positive
definite quadratic form under linear constrains. This application
can be used in many optimization problems, such as electrical
networks
 ( Ben- Israel \cite{Israel2}), finance (Markowitz \cite{Mark1,Mark2} ) etc. A similar result for positive semidefinite quadratic
 forms with many applications in Signal Processing is presented by Stoica et al
\cite{Stoika},  Gorkhov and Stoica \cite{Goro}.\\In this work we
extend the result of Ben- Israel \cite{Israel2} for positive
operators acting on infinite dimensional complex Hilbert spaces. We
will consider the quadratic form as a diagonizable, diagonal or a
positive operator in general, not necessarily invertible. In the
case of a positive semidefinite quadratic form, another approach
proposed for this problem is the constrained minimization to take
place only for the vectors perpendicular to its kernel. This can be
achieved using an appropriate decomposition of the Hilbert space.
\\Another possible candidate for this work would be the class of compact self
adjoint operators, making use of the spectral theorem.
Unfortunately, compact operators do not have closed range, therefore
their generalized inverse is not a bounded operator.

\section{Preliminaries and notation}The notion of the
generalized inverse of a (square or rectangular) matrix was first
introduced by H. Moore in 1920, and again by R. Penrose in 1955.
These two definitions are equivalent, and the generalized inverse of
an operator or matrix is also called the Moore- Penrose inverse. It
is known that when $T$ is singular, then its unique generalized
inverse $T^{\dagger}$ (known as the Moore- Penrose inverse) is
defined. In the case when $T$ is a real $r\times m$ matrix, Penrose
showed that there is a unique matrix satisfying the four Penrose
equations, called the generalized inverse of $T$, noted by
$T^\dagger$.\\In what follows, we consider $\mathcal{H}$ a separable
infinite dimensional Hilbert space and all operators mentioned are
supposed to have closed range.\\The generalized inverse, known as
Moore-Penrose inverse, of an operator $T \in
\mathcal{B}(\mathcal{H})$ with closed range, is the unique operator
satisfying the following four conditions:
\begin{equation}\label{eq-MooPenr}
   TT^\dagger=(TT^\dagger)^*,\qquad T^\dagger T=(T^\dagger T)^*,\qquad
   TT^\dagger T=T,\qquad T^\dagger TT^\dagger =T^\dagger
\end{equation}
where $T^*$ denotes the adjoint operator of $T$.

\noindent It is easy to see that $TT^\dagger$ is the orthogonal
projection of $\mathcal{H}$ onto $\mathcal{R}(T)$ , denoted by
$P_{T}$, and that $T^\dagger T$ is the orthogonal projection of
$\mathcal{H}$ onto $\mathcal{R}(T^*)$ noted by $P_{T^*}$. It is well
known that $\mathcal{R}(T^\dagger)=\mathcal{R}(T^*)$ , and that
$T^\dagger = (T^*T)^{-1}|_{\mathcal{R}(T^*)}T^*$.\\It is also known
that $T^\dagger$ is bounded if and only if $T$ has a closed
range.\\If $T$ has a closed range and  commutes with $T^\dagger$,
then $T$ is called an EP operator. EP operators constitute a wide
class of operators which includes the self adjoint operators, the
normal operators and the invertible operators.
\\Let us consider the equation $Tx=b,T \in B(\mathcal{H})$, where $T$ is singular. If $b\notin R(T)$, then
the equation has no solution. Therefore, instead of trying to solve
the equation $\|Tx-b\|=0$, we may look for a vector $u$ that
minimizes the norm $\|Tx-b\|$. Note that the vector $u$ is unique.
 In this case we consider the equation $Tx=P_{R(T)}b$, where
$P_{R(T)}$ is the orthogonal projection on $\mathcal{R}(T)$.
\\The following two propositions can be found in Groetsch \cite{groe} and hold for operators and matrices:

\begin{proposition}\label{p1}
Let $T \in \mathcal{B}(\mathcal{H})$ and $b\in \mathcal{H}$. Then,
for $u\in \mathcal{H}$, the following are equivalent:
\begin{enumerate}
\item [(i)]$Tu=P_{R(T)}b$
\item [(ii)]$\|Tu-b\| \leq\parallel Tx-b\|, \forall x \in \mathcal{H}$
\item [(iii)]$T^*Tu=T^*b$

\end{enumerate}
\end{proposition}

Let $\mathbb{B}=\{u\in \mathcal{H} |  T^*Tu=T^*b\}$. This set of
solutions is closed and convex, therefore, it has a unique vector
with minimal norm. In the literature, Groetsch\cite{groe},
$\mathbb{B}$ is known as the set of the generalized solutions.

\begin{proposition}\label{p2}
Let $T \in \mathcal{B}(\mathcal{H}), b\in \mathcal{H}$, and the
equation $Tx=b$. Then, if $T^\dag$ is the generalized inverse of
$T$, we have that $T^\dag b = u$, where $u$ is the minimal norm
solution defined above.
\end{proposition}
This property has an application in the problem of minimizing a
symmetric positive definite quadratic form $\langle x,Qx\rangle$
subject to linear constraints, assumed consistent (see Theorem
\ref{th1}).\\We will denote by
 $\mbox{Lat}\,T$ the set of all
closed subspaces of the underlying Hilbert space $\mathcal{H}$
invariant under $T$.\\A self adjoint operator $T \in B(\mathcal{H})$
is positive when $\langle Tx,x\rangle \geq 0$ for all $x \in
\mathcal{H}$. Let $T$ be an invertible positive operator $T$ which
is diagonizable. Then, $T = U^* T_k U$ where $U$ is unitary and
$T_k$ is diagonal, of the form
$$T_k (x_1, x_2, \ldots) = (k_1 x_1, k_2 x_2, \ldots)$$ where $(k_n)_n$ is a bounded sequence of real numbers, and its terms are the eigenvalues of $T_k$, assumed positive.
Its inverse $T_k^{-1}$  is also a diagonal operator, with
corresponding sequence $k'_i = \frac{1}{k_i}$.
\\When $T_k$ is
singular, at least one of the $k_i$'s is equal to zero. Then, its
Moore-Penrose inverse has a corresponding sequence of diagonal
elements $k'_i$ defined as follows:\begin{eqnarray} k'_i=\left \{
\begin{array}{cccc}
\frac{1}{k_i}, &k_i \neq 0
\nonumber \\
0, & k_i =0\nonumber
\end{array} \right .
\end{eqnarray}\\
Since all the diagonal elements are nonnegative, in both cases $T_k$
has a unique square root $T_m$, which is also a diagonal operator
with corresponding sequence $m_n = \sqrt{k_n}$. Similar results
concerning diagonizable and diagonal operators can be found in
Conway  \cite{con}.\\As mentioned before, EP operators include
normal and self adjoint operators, therefore the operator $T$ in the
quadratic form studied in this work is EP. An operator $T$ with
closed range is called EP if $\mathcal{N}(T)=\mathcal{N}(T^*)$. It
is easy to see that
\begin{equation}
\label{def} T \text{
EP}\Leftrightarrow\mathcal{R}(T)=\mathcal{R}(T^*)\Leftrightarrow\displaystyle{\mathcal{R}(T)\mathop{\oplus}^{\perp}\mathcal{N}(T)=\mathcal{H}}\Leftrightarrow
TT^{\dag}=T^{\dag}T.
\end{equation}We take advantage of the fact that EP operators
have a simple canonical form according to the decomposition
$\mathcal{H}= \mathcal{R}(T)\oplus \mathcal{N}(T)$. Indeed an EP
operator $T$ has the following simple matrix form: $$T = \left[
{\begin{array}{*{20}c}
   A & 0  \\
   0 & 0  \\
\end{array}} \right]$$
where the operator $A:\mathcal{R}(T)\rightarrow \mathcal{R}(T)$ is
invertible, and its generalized inverse $T^\dagger$ has the form
$$T^\dagger= \left[ {\begin{array}{*{20}c}
   A^{-1} & 0  \\
   0 & 0  \\
\end{array}} \right]$$ (see Campbell and Meyer \cite{campb}, Drivaliaris et al \cite{Driv}).\\Another result used in our work, wherever a square root of a positive operator is used, is the fact that EP  operators have index  equal to 1 and so, $\mathcal{R}(T) = \mathcal{R}(T^2).$  (see Ben Israel \cite{Israel}, pages 156- 157)\\As mentioned above,
a necessary condition for the existance of a bounded generalized
inverse is that the operator has closed range. Nevertheless, the
range of the product of two operators with closed range is not
always closed.\\In Bouldin \cite{boul2} an equivalent condition is
given:
\begin{theorem}\label{th1} Let $A$ and $B$ be operators with closed range, and let
$$H_i = \mathcal{N}(A)\cap (\mathcal{N}(A)\cap \mathcal{R}(B) )^\bot
 = \mathcal{N}(A) \ominus \mathcal{R}(B)$$ The angle between $H_i$ and $\mathcal{R}(B)$ is positive if and
only if $AB$ has closed range.
\end{theorem}
A similar result can be found in Izumino \cite{Izum}, this time
using orthogonal projections :
\begin{proposition} \label{pr1} Let $A$ and $B$ be operators with closed range. Then, $AB$ has closed
range if and only if $A^\dagger ABB^\dagger$ has closed range.
\end{proposition}
 We will use the above two results to prove the existence
of the Moore- Penrose inverse of appropriate operators which will be
used in our work.
\\Another tool used in this work, is
the reverse order law for the Moore-Penrose inverses. In general,
 the reverse order law does not hold. Conditions under which the reverse order law holds, are described in the following proposition which is a restatement of a part of R. Bouldin's
theorem \cite{boul} that holds for operators and matrices.
\begin{proposition}
Let $A,B$ be bounded operators on $\mathcal{H}$ with closed range.
Then $(AB)^\dagger=B^\dagger A^\dagger $ if and only if the
following three conditions hold:

i) The range of $AB$ is closed,

ii) $A^\dagger A$ commutes with $BB^*$,

iii) $BB^\dagger$ commutes with $A^*A$.
\end{proposition}
 A corollary of the above theorem is the following proposition that can be found
 in Karanasios- Pappas
\cite{kar} and we will use it in our case.
\begin{proposition}\label{pr2}
  Let $A,T\in \mathcal{B}(\mathcal{H})$ be two operators such that $A$
  is invertible and $T$ has closed range. Then
  \[
      (TA)^\dagger= A^{-1}T^\dagger\quad \mbox{if and only if} \quad
      \mathcal{R}(T)\in \mbox{Lat}\,(AA^*).
  \]
\end{proposition}
\section{The Generalized inverse and minimization of Quadratic forms}
Let $Q$ be a symmetric positive definite symmetric matrix. Then, $Q$
can be
written as $Q= UDU^*$, where $U$ is unitary and $D$ is diagonal.\\
Let $D^{\frac{1}{2}}$ denote the positive solution of $X^2 = D$, and
let $D^{-\frac{1}{2}}$ denote $(D^{\frac{1}{2}})^{-1}$, which exists
since $Q$ is positive definite.\\The following theorem can be found
in Ben Israel \cite{Israel2}.
\begin{theorem}\label{th2}Consider the equation $Ax = b$ .\\If the set $S = \{x:Ax = b\}$ is not empty, the the
problem : $$ minimize \langle x,Qx\rangle,  x\in S $$ has the unique
solution $$x= UD^{-\frac{1}{2}}(AU D^{-\frac{1}{2}})^\dagger b.$$
\end{theorem}
\subsection{Positive Diagonizable Quadratic Forms}
A generalization of the above theorem in infinite dimensional
Hilbert spaces, is by replacing $Q$ with an invertible positive
operator $T$ which is diagonizable. The operator $A$ must be
singular, otherwise this problem is trivial. Since $T$ is
diagonizable, we have that $T = U^* T_k U$.\\We need first the
following Lemma. Note that in the infinite dimensional case the
Moore-Penrose inverse of an operator is bounded if and only if the
operator has closed range.
\begin{lemma} \label{l1} Let $T \in \mathcal{B}(\mathcal{H})$ be an invertible positive
operator with closed range
 which is diagonizable and  $A\in \mathcal{B}(\mathcal{H})$
singular with closed range.\\Then, the range of $AU^* X^{-1}$ is
closed , where $X$ is the unique solution of the equation $X^2 =
T_k$.\begin{proof} Using Theorem \ref{th1}, the range of $U^*
X^{-1}$ is closed since $H_i =\mathcal{N}(U^*) \ominus
\mathcal{R}(X^{-1}) = {0}$, and so $H_i$ is trivial, therefore the
angle between $U^*$ and $X^{-1}$ is equal to $\frac{\pi}{2}$. Hence,
the range of $AU^* X^{-1}$ is closed because using proposition \ref
{pr1} we can see that $A^\dagger A (U^* X^{-1})(U^* X^{-1})^\dagger
= P_{A^*} U^* X^{-1} XU = P_{A^*}$ has closed range.
\end{proof}
\end{lemma}
We are now in condition to prove Theorem \ref{th3}.
\begin{theorem}\label{th3} Consider the equation $Ax = b$, with $A\in \mathcal{B}(\mathcal{H})$
singular with closed range and $b \in \mathcal{H}$.\\ If the set $S
= \{x:Ax = b\}$ is not empty, then the problem :
$$ minimize \langle x,Tx\rangle, x\in S $$
with $T \in \mathcal{B}(\mathcal{H})$ an invertible positive
diagonizable operator with closed range has the unique solution
$$\hat{x} = U^*X^{-1}(AU^* X^{-1})^\dagger b$$ where $X$ is the
unique solution of the equation $X^2 = T_k$.
\end{theorem}
\begin{proof}The idea of the proof is similar to Ben Israel \cite{Israel2}, but the existence of a bounded Moore- Penrose inverse is not trivial like in the finite dimensional case.
\\It is easy to see that since $T = U^* T_k U$
 is positive, $T_k$ is also positive. Then,
$$\langle x,Tx\rangle = \langle x,U^*T_k Ux\rangle = \langle Ux,T_k
Ux\rangle = \langle Ux,X^2 Ux\rangle =$$ $$ \langle XUx,XUx\rangle =
\langle y,y\rangle$$
\\So, the problem of minimizing $\langle x,Tx\rangle$ is equivalent
of minimizing $\langle y,y\rangle =
\parallel y\parallel^2$ where $y =XUx$. We also have that $y =XUx
\Leftrightarrow x = U^* X^{-1} y$.\\ From proposition \ref{p2} we
know that the minimal norm solution of the equation $Ax = b$ is the
vector $\hat{x} = A^\dagger b$, so by substitution, $AU^* X^{-1} y =
b$ and the minimal norm vector $\hat{y}$ is equal to $\hat{y}= [AU^*
X^{-1}]^\dagger b$, since $AU^* X^{-1}$ is a singular
operator.\\Therefore, $XU\hat{x}= [AU^* X^{-1}]^\dagger b
\Leftrightarrow \hat{x}= U^*X^{-1}(AU^* X^{-1})^\dagger b$

\end{proof}
We can verify that the solution $\hat{x} = U^*X^{-1}(AU^*
X^{-1})^\dagger b$ satisfies the constraint $Ax = b$ :\\We have that
$A \hat{x} = AU^*X^{-1}(AU^* X^{-1})^\dagger b = SS^\dagger b= P_S
b$, where $S = AU^* X^{-1}$ and we can see that $\mathcal{R}(S) =
\mathcal{R}(AU^* X^{-1}) = \mathcal{R}(A)$ since $X$ and $U$ are
invertible. Therefore, $P_S = P_A $ and $P_S b = P_A b =b$ since $b
\in\mathcal{R}(A)$.\\We can also compute the value of the minimum
$\langle x,Tx\rangle, x\in S $ : $$\langle \hat{x},T\hat{x}\rangle
 = \langle U^*X^{-1}(AU^* X^{-1})^\dagger b, U^*T_kUU^*X^{-1}(AU^* X^{-1})^\dagger
 b \rangle =$$ $$ \langle (AU^* X^{-1})^\dagger b, (AU^* X^{-1})^\dagger
 b \rangle = \parallel (AU^*X^{-1})^\dagger b \parallel^2$$
\subsection{Positive Definite Quadratic Forms}
In this section we extend the results presented above, in the
general case when $T$ is a positive operator following the same
point of view.\\Let $T$ be a positive operator, having a unique
square root $R$.
\begin{theorem}\label{th4}Consider the equation $Ax = b$, with $A\in \mathcal{B}(\mathcal{H})$
singular with closed range and $b \in \mathcal{H}$. If the set $S =
\{x:Ax = b\}$ is not empty, then the problem :
$$ minimize \langle x,Tx\rangle, x\in S $$ with $T \in
\mathcal{B}(\mathcal{H})$ an invertible positive operator with
closed range has the unique solution
$$\hat{x} = R^{-1}(AR^{-1})^\dagger b$$
\end{theorem}
\begin{proof}
We have that $\langle x,Tx\rangle = \langle Rx,Rx\rangle =
\parallel  Rx \parallel^2=
\parallel y\parallel^2$. So, $$Ax = b \Leftrightarrow \hat{y} =(AR^{-1})^\dagger
b \Leftrightarrow \hat{x} = R^{-1}(AR^{-1})^\dagger b$$ It is easy
to verify that the range of the operator $ AR^{-1}$ is closed, as in
lemma \ref{l1}.
\end{proof}\\We can see by easy computations, that the minimum  $\langle
x,Tx\rangle, x\in S $ is  then equal to  $$\langle \hat{x},T
\hat{x}\rangle =\langle R^{-1}(AR^{-1})^\dagger b,R^2
R^{-1}(AR^{-1})^\dagger b \rangle =$$ $$ \langle (AR^{-1})^\dagger
b, (AR^{-1})^\dagger b \rangle =
\parallel (AR^{-1})^\dagger b \parallel^2 $$

\noindent A natural question to ask, is what happens if the reverse
order law for generalized inverses holds. In this case, the solution
given by Theorem \ref{th4}, using Proposition \ref{pr1} will be as
follows:
\begin{corollary}\label{cor1} Considering all the assumptions of Theorem
\ref{th4}, let  $\mathcal{R}(A)\in \mbox{Lat}\,(T)  $. Then,
$\hat{x} = A^\dagger b$.
\end{corollary}The proof is obvious, since in this case, $Lat(RR^*)=  Lat(R^2) = Lat(T)$, and  $R^{-1}(AR^{-1})^\dagger
 = R^{-1}R A^\dagger $\\We can also see that in this case, the minimum value of $\langle
x,Tx\rangle, x\in S $ is equal to $\langle A^\dagger b,R^2 A^\dagger
b\rangle= \parallel RA^\dagger b \parallel^2 $\\We will present an
example for Theorem \ref{th4} and corollary \ref{cor1}.
\begin{example}
 Let $T:l_2 \rightarrow l_2  : T(x_1, x_2, x_3, \ldots) =
(x_1, 2x_2, x_3, 2x_4 \ldots)$ which is a bounded diagonal linear
operator.\\Let $L : l_2\rightarrow l_2  : L(x_1, x_2, x_3, \ldots) =
(x_2, x_3, \ldots)$, the well known left shift operator which is
singular and $S = \{x:Lx =
(1,\frac{1}{2},\frac{1}{3},\frac{1}{4},\ldots)\}$. It is also well
known that $L^\dagger = R$, the right shift operator. Since $T$ is
positive and invertible, the problem of minimizing $\langle
x,Tx\rangle$ , $ x \in S $ following Corollary \ref{cor1} has the
unique solution $\hat{x} =
(0,1,\frac{1}{2},\frac{1}{3},\frac{1}{4},\ldots )$.\\Indeed, since
$\mathcal {R} (L) = \mathcal{H}= l_2$ which is invariant under $T$,
we have that $\mathcal {R} (L)\in Lat(T)$ and so
$$\hat{x} = (L)^\dagger (1,\frac{1}{2},\frac{1}{3},\frac{1}{4},
\ldots) = (0,1,\frac{1}{2},\frac{1}{3},\frac{1}{4},\ldots )$$ Using
this vector, we have that the problem of minimizing $ \langle
x,Tx\rangle, x\in S $ has a minimum value as shown in what follows:
$$min \langle x,Tx\rangle = \langle \hat{x},T\hat{x}\rangle = \langle
(0,1,\frac{1}{2},\frac{1}{3},\frac{1}{4},\ldots )
,T(0,1,\frac{1}{2},\frac{1}{3},\frac{1}{4},\ldots )\rangle = $$
$$\langle (0,1,\frac{1}{2},\frac{1}{3},\frac{1}{4},\ldots ),(0,2,\frac{1}{2},2\times\frac{1}{3}, \frac{1}{4}, 2\times\frac{1}{5}\ldots
)\rangle = 0+2\times1 +\frac{1}{2^2}+2\times\frac{1}{3^2}+
\frac{1}{4}+ 2\times\frac{1}{5^2} \ldots$$  $$= \sum_{n=1}^\infty
\frac{1}{n^2} + \sum_{n=0}^\infty \frac{1}{(2n+1)^2}=\frac{\pi^2}{6
}+ \frac{\pi^2}{8}= \frac{7 \pi^2}{24}$$ This value is equal to
$\parallel RA^\dagger b
\parallel^2 $ as presented in the above corollary, since in this
case $R(x_1, x_2, x_3, \ldots) = (x_1, \sqrt{2}x_2, x_3, \sqrt{2}x_4
\ldots)$ and
$$\parallel RA^\dagger b
\parallel^2  = \parallel RL^\dagger
(1,\frac{1}{2},\frac{1}{3},\frac{1}{4},\ldots)\parallel^2 =
0+2\times1 +\frac{1}{2^2}+2\times\frac{1}{3^2}+ \frac{1}{4}+\ldots =
\frac{7 \pi^2}{24}$$ and this verifies Corollary \ref{cor1}.\\We can
see that the minimizing vector found by Theorem \ref{th4} has the
minimum norm among all possible solutions, of the form
$(c,1,\frac{1}{2},\frac{1}{3},\frac{1}{4},\ldots ), c \in
\mathcal{C}$ , as expected.
\end{example}
 \subsection{Positive Semidefinite Quadratic Forms}
We can also consider the case when the positive operator $T$ is
singular, that is, $T$ is positive semidefinite. In this case, since
$ \mathcal{N}(T) \neq\emptyset$,  we have that $\langle x,Tx\rangle
= 0$ for all $x \in \mathcal{N}(T)$ and so, the problem :
$$ minimize \langle x,Tx\rangle, x\in S $$ has many solutions when $\mathcal{N}(T)\cap S \neq\emptyset$.\\In Stoika et al \cite{Stoika}
a method is presented for the minimization of a positive
semidefinite quadratic form under linear constraints, with many
applications in the finite dimensional case. In fact, since this
problem has an entire set of solutions, the minimum norm solution is
given explicitly.
\\A different approach to this problem in both the finite and
infinite dimensional case would be to look among the vectors $ x \in
\mathcal{N}(T)^\perp = \mathcal{R}(T^*) = \mathcal{R}(T)$ for a
minimizing vector for $\langle x,Tx\rangle $. In other words, we
will look for the minimum of $\langle x,Tx\rangle$ under the
constraints $Ax = b, x \in \mathcal{R}(T).$\\Using the fact that $T$
is an $EP$ operator, we will make use of the first two conditions in
the following proposition that can be found in Drivaliaris et al
\cite{Driv}:
\begin{proposition}
\label{prop5} Let $T\in\mathcal{B}(\mathcal{H})$ with closed range.
Then the following are equivalent: \\i) $T$ is EP. \\ii) There exist
Hilbert spaces $\mathcal{K}_1$ and $\mathcal{L}_1$,
$U_1\in\mathcal{B}(\mathcal{K}_1\oplus \mathcal{L}_1,\mathcal{H})$
unitary and $A_1\in\mathcal{B}(\mathcal{K}_1)$ isomorphism such that
$$T=U_1(A_1\oplus 0)U_1^*.$$
iii) There exist Hilbert spaces $\mathcal{K}_2$ and $\mathcal{L}_2$,
$U_2\in\mathcal{B}(\mathcal{K}_2\oplus \mathcal{L}_2,\mathcal{H})$
isomorphism and $A_2\in\mathcal{B}(\mathcal{K}_2)$ isomorphism such
that
$$T=U_2(A_2\oplus 0)U_2^*.$$
iv) There exist Hilbert spaces $\mathcal{K}_3$ and $\mathcal{L}_3$,
$U_3\in\mathcal{B}(\mathcal{K}_3\oplus \mathcal{L}_3,\mathcal{H})$
injective and $A_3\in\mathcal{B}(\mathcal{K}_3)$ isomorphism such
that
$$T=U_3(A_3\oplus 0)U_3^*.$$
\end{proposition}
We present a sketch of the proof for (1)$\Rightarrow$(2):
\begin{proof}
Let $\mathcal{K}_1=\mathcal{R}(T)$, $\mathcal{L}_1=\mathcal{N}(T)$,
$U_1:\mathcal{K}_1\oplus\mathcal{L}_1\rightarrow\mathcal{H}$ with
$$U_1(x_1,x_2)=x_1+x_2,$$ for all $x_1\in\mathcal{R}(T)$ and
$x_2\in\mathcal{N}(T)$, and
$A_1=T|_{\mathcal{R}(T)}:\mathcal{R}(T)\rightarrow\mathcal{R}(T).$
Since $T$ is EP,
$\mathcal{R}(T)\mathop{\oplus}^{\perp}\mathcal{N}(T)=\mathcal{H}$
and thus $U_1$ is unitary. Moreover it is easy to see that
$U_1^*x=(P_Tx,P_{\mathcal{N}(T)}x),$ for all $x\in \mathcal{H}$. It
is obvious that $A_1$ is an isomorphism. A simple calculation shows
that
$$T=U_1(A_1\oplus 0)U_1^*.$$
\end{proof}\\It is easy to see that when $T=U_1(A_1\oplus 0)U_1^*$ and $T$ is
positive, so is $A_1$, since $\langle x,Tx\rangle = \langle
x_1,A_1x_1\rangle, x_1 \in \mathcal{R}(T)$.
\\ In what follows, $T$
will denote a singular positive operator with a canonical form
$T=U_1(A_1\oplus 0)U_1^*$ , $R$ is the unique solution of the
equation $R^2 = A_1$ and $$R^\dagger = \left[ {\begin{array}{cc}
   R^{-1} & 0  \\
   0 & 0  \\
\end{array}} \right]$$\\ As in the previous cases, since the two
operators $A$ and $R$ are arbitrary, one does not expect that the
range of their product will always be closed.
\\Using Proposition
\ref{prop5},  we have the following theorem:
\begin{theorem}\label{th6} Let $T \in \mathcal{B}(\mathcal{H})$ be an singular positive
operator,  and the equation $Ax = b$, with $A\in
\mathcal{B}(\mathcal{H})$ singular with closed range and $b \in
\mathcal{H}$. If the set $S = \{x \in \mathcal{N}(T)^\perp :Ax =
b\}$ is not empty, then the problem :
$$ minimize \langle x,Tx\rangle, x\in S $$ has the unique solution
$$ \hat{x}=
U_1R^\dagger(AU_1 R^\dagger)^\dagger b$$ assuming that $P_{A^*}P_T$
has closed range.
\end{theorem}
\begin{proof}We have that
$$\langle x,Tx\rangle = \langle x,U_1(A_1\oplus 0)U_1^*x\rangle =
\langle U_1^*x,(A_1\oplus 0) U_1^*x\rangle = \langle
U_1^*x,(R^2\oplus 0) U_1^*x\rangle $$ We have that $U_1^*x =(x_1,
x_2)$ and $\langle U_1^*x,(A_1\oplus 0) U_1^*x\rangle = \langle
x_1,A_1x_1\rangle, x_1 \in \mathcal{R}(T).$\\Therefore $\langle
x,Tx\rangle= \langle (R\oplus 0)U_1^*x,(R\oplus 0) U_1^*x\rangle =
\langle Rx_1,R x_1\rangle=\langle y,y\rangle$ , where $y = Rx_1$,
with $x_1 \in \mathcal{N}(T)^\perp$.
\\The problem of minimizing $\langle x,Tx\rangle$ is equivalent of
minimizing $\parallel y\parallel^2$ where $y = Rx_1=(R\oplus
0)U_1^*x \Leftrightarrow x = U_1 (R^{-1}\oplus 0) y = U_1 R^\dagger
y$.\\As before, the minimal norm solution $\hat{y}$ is equal to
$\hat{y}= [AU_1 R^\dagger]^\dagger b$.\\Therefore, $ \hat{x_1}=
U_1R^\dagger(AU_1 R^\dagger)^\dagger b$, with $\hat{x_1} \in S$.\\As
in Theorem \ref {th3} ,
we still have to prove that $ AU_1 R^\dagger$ has closed range.\\
Using Theorem \ref{th1}, the range of $U_1 R^\dagger$ is closed
since $$H_i =\mathcal{N}(U_1^*)\cap (\mathcal{N}(U_1^*)\cap
\mathcal{R}(R^\dagger) )^\bot = {0}$$ and so
 the angle between $U_1^*$ and $R^\dagger$ is
equal to $\frac{\pi}{2}$.\\From Proposition \ref{pr1} the operator
$P_{A^*}P_T$ must have closed range because $$A^\dagger
AU_1R^\dagger(U_1R^\dagger)^\dagger = P_{A^*}U_1 P_R U_1^* =
P_{A^*}U_1 P_{A_1} U_1^*= P_{A^*}P_T$$ making use of Proposition
 \ref{pr2} and the fact that $\mathcal{R}(R) = \mathcal{R}(A_1)= \mathcal{R}(T) $.

\end{proof}
\begin{corollary}\label{cor2}Under all the assumptions of Theorem
\ref{th6} we have that the minimum value of $f(x) =\langle
x,Tx\rangle , x\in S$ is equal to $\parallel (AU_1R^\dagger)^\dagger
b
\parallel^2 $
\end{corollary}
\begin{proof}
We have that $$f_{min}(x) = \langle \hat{x},T\hat{x}\rangle =
\langle U_1R^\dagger(AU_1 R^\dagger)^\dagger b, TU_1R^\dagger(AU_1
R^\dagger)^\dagger b \rangle$$ Since $T=U_1(R^2\oplus 0)U_1^*$ we
have that $$f_{min}(x) = \langle U_1 R^\dagger (AU_1
R^\dagger)^\dagger b, U_1(R\oplus 0)(AU_1 R^\dagger)^\dagger
b\rangle=$$  $$\langle P_T (AU_1 R^\dagger)^\dagger b, (AU_1
R^\dagger)^\dagger b\rangle = \parallel (AU_1R^\dagger)^\dagger b
\parallel^2 $$
since $$R^\dagger(R\oplus 0) = (I\oplus 0)= P_T$$ and $
\mathcal{R}(AU_1R^\dagger)^\dagger = \mathcal{R}(AU_1R^\dagger)^* =
\mathcal{R}(RU_1 A^*)\subseteq \mathcal{R}(R) = \mathcal{R}(T) $,
therefore
$$P_T (AU_1 R^\dagger)^\dagger b = (AU_1 R^\dagger)^\dagger b$$
\end{proof}\\
In the sequel, we present an example which clarifies Theorem
\ref{th6} and Corollary \ref{cor2}. In addition, the difference
between the proposed minimization ($x \in \mathcal{N}(T)^\perp$) and
the minimization for all $x \in \mathcal{H}$ is clearly indicated.
\begin{example}
Let $\mathcal{H} = \mathcal{R}^3$, and the positive semidefinite
matrix $$Q = \left[ {\begin{array}{ccc}
   14 & 20 & 28  \\
   20 & 83 & 40  \\
   28 & 40 & 56
\end{array}} \right]$$ We are looking for the minimum of $f(u) = u' Q u, u\in \mathcal{N}(Q)^\perp$ under the constraint $2x+ 2y
-z=10.$\\ Then, all vectors $u \in \mathcal{N}(Q)^\perp$ have the
form $ u = (x, y, 2x)$. The matrices $U, R^\dagger$ are $$U = \left[
{\begin{array}{ccc}
   -0.2926  &  0.3382 &  -0.8944  \\
   -0.7563  & -0.6543   & 0  \\
   -0.5852  &  0.6764  &  0.4472
\end{array}} \right]\qquad R^\dagger = \left[
{\begin{array}{ccc}
   0.0907    &     0    &     0  \\
   0  &  0.1787      &   0  \\
  0  &  0  &  0
\end{array}} \right]$$
Using theorem \ref{th6} we see that the minimizing vector of $f(u)$
under $Au=b, u \in\mathcal{N}(Q)^\perp$, where $ A= \left[
{\begin{array}{ccc}
   2&  1 & -1
\end{array}} \right]$ and b = 10, is $$\hat{u} = U_1R^\dagger(AU_1 R^\dagger)^\dagger b = (-2.8572, 10 ,
-5.7143)$$\\The minimum value of $f(u)$ is then equal to 5442.9\\We
can verify that it is equal to the minimum value found in Corollary
\ref{cor2}.
$$\parallel (AU_1R^\dagger)^\dagger b
\parallel^2 =\parallel (
   -37.313,  -63.644 ,  0) \parallel^2=
5442.9$$
\\This example can be represented graphically as follows, clearly showing the constrained minimization and the uniqueness of the solution:
\begin{center}
\begin{figure}[h!] \label{f1}
\includegraphics[width=2.8in,height=2.5in]{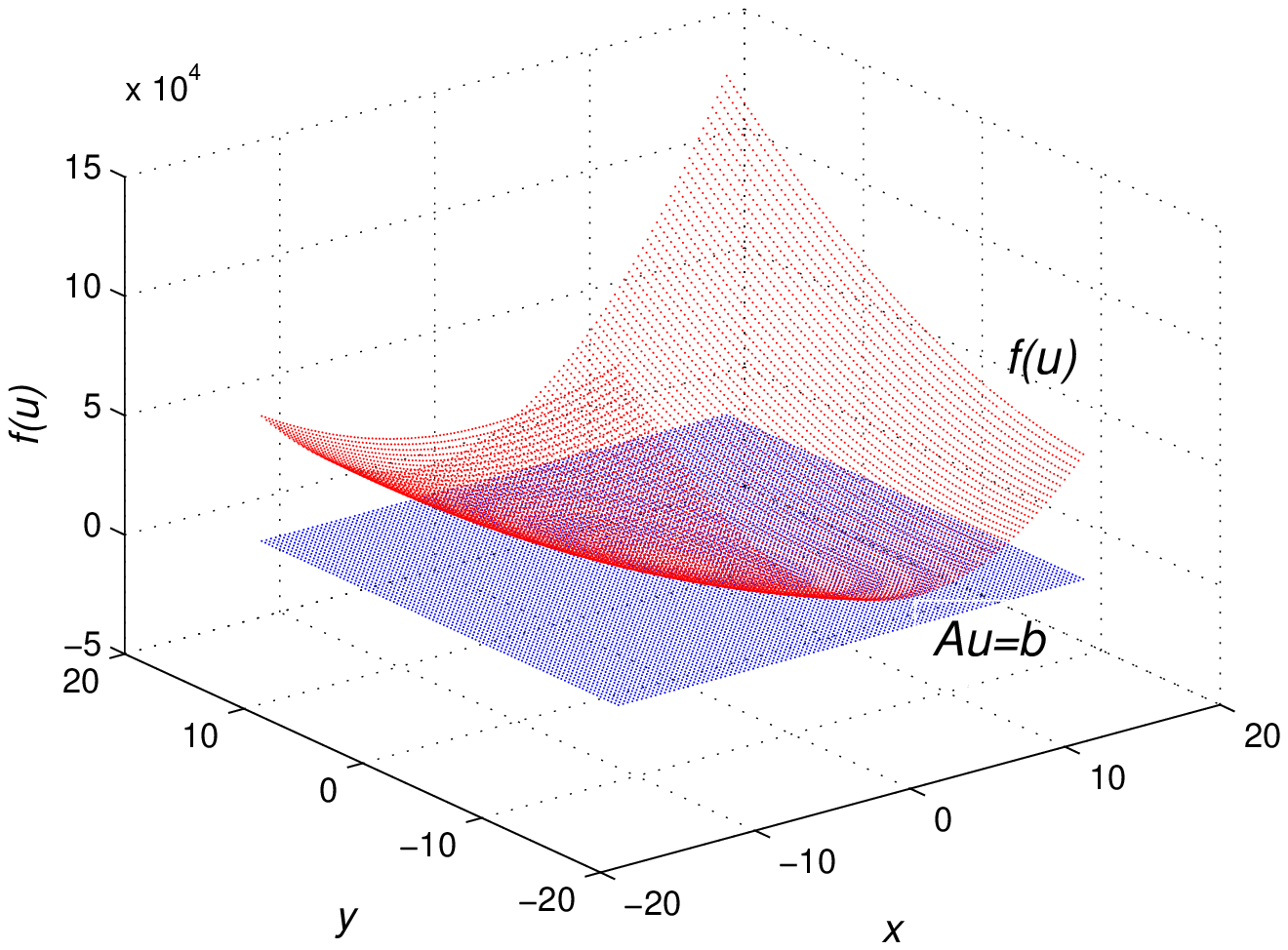}
\includegraphics[width=2.6in,height=2.1in]{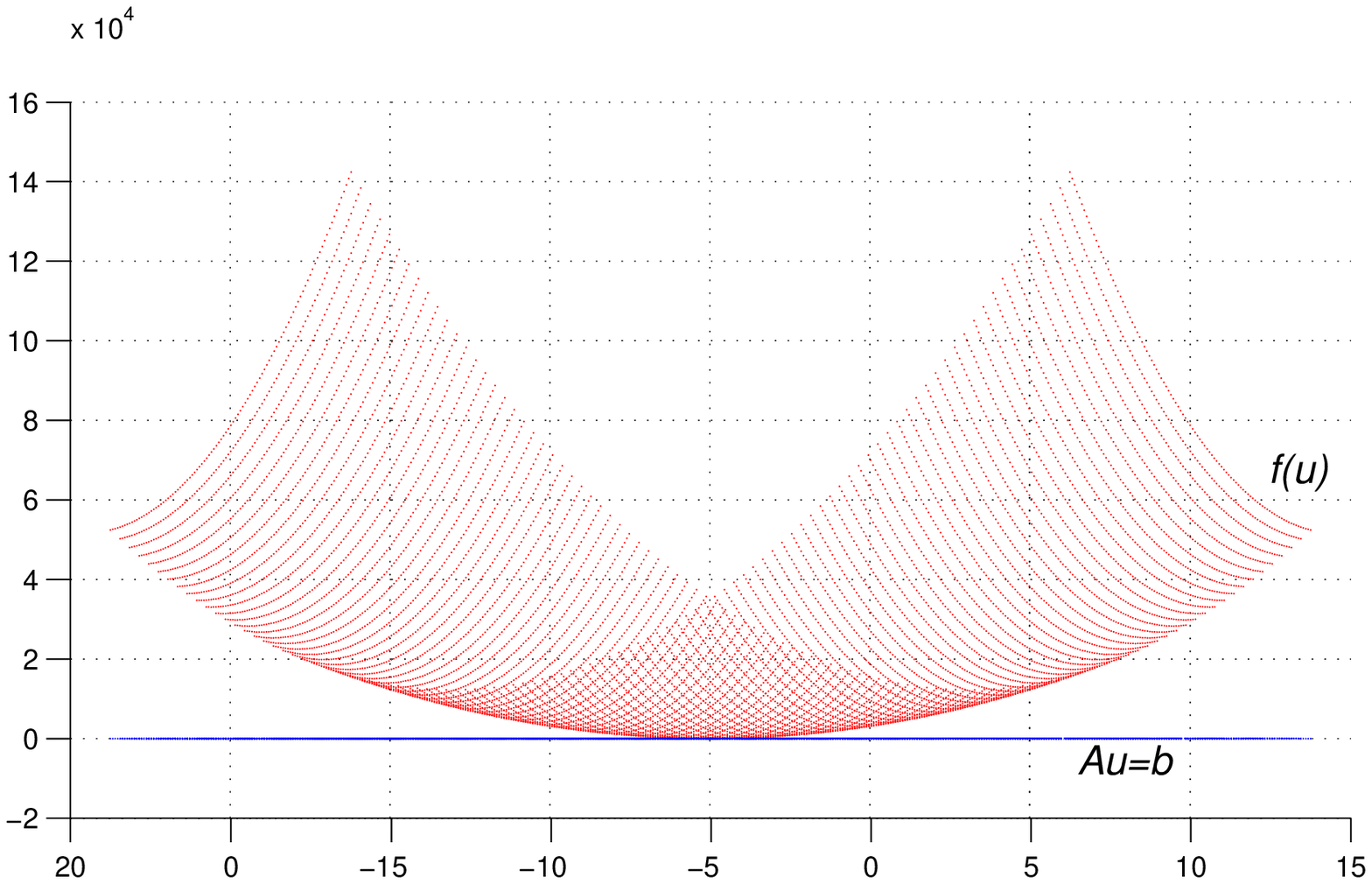}
\caption{Constrained minimization of u'Qu, $u
\in\mathcal{N}(Q)^\perp$ }
\end{figure}
\end{center}
In the case of the minimization of $f(u), u\in \mathcal{R}^3 : Au=
b$, then the set  $$S' = \{u \in \mathcal{N}(Q) :Au = b\}$$ is
nonempty.\\The vector $ v= (4, 0, -2)$ belongs to $ \mathcal{S'}$
and therefore, $f(v)= 0$. The same answer would be given using the
algorithm proposed by Stoika et al \cite{Stoika}.
\end{example}
\section{Conclusions}
In this work we extend a minimization result concerning  non
singular quadratic forms using the Moore-Penrose inverse, to
infinite dimensional Hilbert spaces. In addition, in the case of a
singular quadratic form the minimization takes place for all its non
zero values. This proposed Constrained minimization method has the
advantage of a unique solution and is easy to implement. Practical
importance of this result can be in numerous applications such as
filter design, spectral analysis, direction finding etc. In many of
these cases the quadratic form may be very close to, or even exactly
singular, and therefore the knowledge of the non zero part of the
solution may be of importance.


\end{document}